\documentclass[12pt]{amsart}
\usepackage{comment}
\usepackage{macros}

\newtheorem*{theorem*}{Theorem}
\newtheorem{maintheorem}{Theorem}

\newcommand{\catA}{\mathcal{A}}

\newcommand{\Kernel}{\operatorname{Ker}}

\newcommand{\lra}{\longrightarrow}

\newcommand{\quadand}{\quad\text{and}\quad}

\newcommand{\ra}{\rightarrow}

\newcommand{\val}{\operatorname{val}}

\begin{document}
\title
[The resolution property holds away from codimension three] 
{The resolution property holds away from codimension three}
\date{\today}
\author{Siddharth Mathur}
\email{https://sites.google.com/view/sidmathur/home} 
\address{Mathematisches Institut\\Heinrich-Heine-Universit\"at\\40204 D\"usseldorf, Germany.\\}

\author{Stefan Schr\"oer}
\email{schroeer@math.uni-duesseldorf.de} 
\address{Mathematisches Institut\\Heinrich-Heine-Universit\"at\\40204 D\"usseldorf, Germany.\\}

\begin{abstract}

\noindent The purpose of this paper is to verify a conjecture of Gross under mild hypotheses: all reduced, separated, and excellent schemes have the resolution property away from a closed subset of codimension $\geq 3$. Our technique uses formal-local descent and the existence of affine flat neighborhoods to reduce the problem to constructing certain modules over commutative rings. Once in the category of modules we exhibit enough locally free sheaves directly, thereby establishing the resolution property for a specific class of algebraic spaces. A crucial step is showing it suffices to resolve a single coherent sheaf.

\end{abstract}
\maketitle

%%%%%%%%%%%%%%%%%%%%%%%%%%%%%%%%%%%%%%%%%%%%%%

\tableofcontents

%%%%%%%%%%%%%%%%%%%%%%%%%%%%%%%%%%%%%%%%%%%%%%

\section{Introduction}
\label{sec:intro}

A scheme is said to have the \emph{resolution property} if every quasi-coherent sheaf of finite type is the quotient of a locally free sheaf of finite rank. Although quasi-projective and regular schemes have enough locally free sheaves, very little is understood in the broader category of separated schemes or algebraic spaces. Already for surfaces this is a difficult problem (see \cite{NormalSurfaceRes}, \cite{SurfaceRes}, and \cite{SidRes}) and the case of threefolds seems intractable: it is not even known if a proper threefold always admits a single nontrivial locally free sheaf. Although interesting progress has been made in the toric case (see \cite{MR2448277} and \cite{PerlingToric}), a general solution remains elusive even for this simpler question. 

In his thesis, Gross conjectured that separated schemes enjoy the resolution property away from a closed subset of codimension $\geq 3$ (see \cite[Conj. 2.0.3]{Grossthesis}). Among other things, this is an important step for establishing the resolution property for threefolds. In this paper, we verify this conjecture under mild hypotheses:
\begin{maintheorem} \label{maintheorem} Let $X$ denote a scheme which is excellent, reduced and separated. 
\begin{enumerate} \item Every point $x \in X$ admits an open neighborhood $U \subset X$ that contains all points of codimension $\leq 2$ and has the resolution property. 
\item Every coherent sheaf $\sh{F}$ on $X$ is the quotient of a coherent sheaf $\sh{E}$ which is locally free at all points of codimension $\leq 2$. \end{enumerate} \end{maintheorem}

\noindent In fact, our arguments yield stronger results, many of which cover algebraic spaces, see Theorem, \ref{thm:singlecoherent}, Theorem \ref{resolutionthree}, and Corollary \ref{surjectionthree} for more precise statements. 

The existence of locally free resolutions on schemes is a well studied question and perhaps the most striking result in this direction is the following result due to Totaro and Gross.

\begin{theorem*} (\cite[Thm. 1.1]{TotaroRes} and \cite[Thm. 1.1]{GrossRes}) A quasi-compact and quasi-separated  algebraic space $X$ has the resolution property if and only if $X$ can be written as the quotient of a free action of $\mathrm{GL}_n$ on a quasi-affine scheme. \end{theorem*}

\noindent As such, the resolution property generalizes the notion of quasi-projectivity: varieties embedded in $\mathbb{P}^n$ always arise as quotients of quasi-affine schemes by a free action of $\GL_1$. Moreover, a defining property of ample line bundles is that they (positively) generate all the coherent sheaves on a space. Since this property is so commonly used, it is no surprise that the resolution property has important and fundamental consequences across different areas of algebraic geometry: Hilbert's fourteenth problem over a general base (see \cite[Thm. 2]{Seshadri} and \cite[Thm 3.8]{ThomasonRes}), the moduli of polarized varieties (see \cite[Ques. 43]{MR2259254} and \cite[Conj. 42]{MR2259254}), the theory of group schemes (\cite{ConradDual}), the Brauer group (\cite{SidRes}), and algebraic K-theory (\cite[Sec. 2]{TotaroRes}). This justifies the importance of the following question of Totaro. It was originally stated for algebraic stacks but is still tantalizingly open in the absence of stabilizers: 
\\\\
\noindent \textbf{Totaro's question} (Totaro, \cite[Ques. 1]{TotaroRes}) Let $X$ be a separated scheme (or algebraic space), does $X$ have the resolution property? 
\\\\
Despite the utility of the resolution property, little progress has been made beyond algebraic surfaces. Even the case of smooth three dimensional algebraic spaces over a field is completely wide open and the problem seems to be beyond the reach of current methods. Before describing our strategy, we recall how to resolve a coherent sheaf $\sh{F}$ on a surface $X$:\begin{enumerate} 

\item First, we realize $\sh{F}$ as a quotient of a coherent sheaf which is locally free in codimension $\leq 1$ (see \cite[Prop. 1.8]{SurfaceRes}). 
The crucial idea here is that, using Chow's lemma, one can find a quasiprojective open subset containing all height one points. Now we may assume $\sh{F}$ is locally free away from finitely many points $p_1,...,p_n$ and so it suffices to extend local resolutions at the $p_i$ to all of $X$.

\item This extension problem is solved by constructing locally free sheaves with certain positivity properties (see \cite[Thm. 4.5]{SurfaceRes}). The idea is to use Chow's lemma to produce a projective birational cover and one then descends a positive bundle by modifying a sum of ample line bundles along a $1$-dimensional exceptional locus. To conclude, one annihilates the obstruction to extending these local resolutions by twisting with the aforementioned locally free sheaves. \end{enumerate}

The first step automatically fails in our setting because the codimension $2$ points need not admit a quasi-projective open neighborhood. Thus, we have to find a tractable replacement for such opens. Our solution is to algebraize formal-local neighborhoods of any fixed finite collection of codimension $2$ points, this process yields affine \emph{flat neighborhoods} (see Definition \ref{def:flatneighborhood}) and when glued along an affine \emph{open} subset of $X$, we obtain open subspaces $U_i \subset X$ which contains this finite collection of codimension $2$ points (see Proposition \ref{prop:affineMV}). Moreover recent results on flat descent allow us to reduce the problem on the various $U_i$ to the following result involving only modules over commutative rings (see Theorem \ref{resolution property from commutative algebra}). %Before stating the result we introduce some notation: if $R_1$ and $R_2$ are commutative rings then the $2$-fiber product $\mathrm{Mod}(R_1) \times_{\mathrm{Mod}(R_{12})} \mathrm{Mod}(R_2)$ is the category consisting of triples $(M_1,M_2,\varphi)$ where $M_i$ is an $R_i$-module and $\varphi:M_1 \otimes_{R_1} R_{12} \xrightarrow{\sim} M_2 \otimes_{R_2} R_{12}$ is an isomorphism. A morphism $(M_1,M_2,\varphi) \to (N_1,N_2,\varphi')$ is a pair of maps $f_i:M_i \to N_i$ satisfying $(f_2 \otimes \text{id}_{R_{12}}) \circ \varphi' = \varphi \circ (f_1 \otimes \text{id}_{R_{12}})$.

\begin{maintheorem}
\label{intro: resolution property from commutative algebra}
Suppose we have ring maps $R_i\ra R_{12}$ which are flat for $i=1,2$ with the property that the multiplication map $R_{12}\otimes_{R_2}R_{12}\ra R_{12}$ is bijective. Then every object $(M_1,M_2,\varphi) \in \mathrm{Mod}(R_1) \times_{\mathrm{Mod}(R_{12})} \mathrm{Mod}(R_2)$, where the two modules $M_1$ and $M_2$ are finitely generated modules and $\varphi:M_1 \otimes_{R_1} R_{12} \xrightarrow{\sim} M_2 \otimes_{R_2} R_{12}$ is an isomorphism, is the quotient of the triple $(R_1^{\oplus n}, R_2^{\oplus n}, A)$ for a suitable $n\geq 0$ and $A\in \GL_n(R_{12})$.
\end{maintheorem} 

The second step fails in our context for two reasons: first, we work without properness assumptions so cohomological vanishing arguments need not apply. The second issue, even if we only consider proper $k$-schemes, is that we do not know how to descend a desirable locally free sheaf along a higher dimensional projective birational cover. In particular, recall that it is not even known how to descend a single non-trivial locally free sheaf for the case of proper threefolds. We evade these issues entirely by delicately constructing the open subsets $U_i$, then pushing forward local surjections $\sh{V}_i \to \sh{F}|_{U_i}$ and taking their direct sum. To ensure that the push-forwards retain the essential information of $\sh{F}$, we also choose the opens $U_i$ so they contain most of the non-Cohen-Maucalay locus of $\sh{F}$, then depth considerations finish the argument.

\begin{Remark} It is natural to ask to what extent Theorem \ref{maintheorem} holds true for algebraic stacks. There are $\mathbb{G}_m$-gerbes over normal separated surfaces which do not enjoy the resolution property (see \cite[Thm. 1.1]{TotaroRes} and \cite[Ex. 3.12]{EHKV}), so the result cannot hold for all stacks with affine diagonal. However, it would be interesting to know if Theorem \ref{maintheorem} holds for algebraic stacks with \emph{finite} diagonal. Some special cases in this direction are known (see \cite{SidRes} and \cite{KreschVistoli}) but the general case seems to be out of reach. Interestingly, the existence of affine flat neighborhoods is a crucial ingredient in the present paper and it has also been used to prove Theorem \ref{maintheorem} for $\mu_n$-gerbes (see \cite[Cor. 2]{mathur2020experiments} and its proof). However, the method we use to construct locally free resolutions (see Theorem \ref{intro: resolution property from commutative algebra}) only works for \emph{algebraic spaces} of a special form (see Corollary \ref{MVresolution}). \end{Remark}

\noindent \textbf{Structure of the paper}: In the Section \ref{sec:common} we establish a necessary refinement of the fact that any finite collection of codimension $\leq 1$ points admits a common affine open neighborhood. Since this is not adequate for codimension $2$ points, we show how to improve this result further using affine flat neighborhoods in Section \ref{Flat neighborhoods}. Then we show that these new neighborhoods have the resolution property using an algebraic argument in Section \ref{Triples}. In Section \ref{resolutions}, we prove that on a very general class of algebraic spaces, any coherent sheaf can be written as the quotient of a locally free sheaf on an open set containing all codimension $\leq 2$ points. In Section \ref{strong}, we explain why it is often enough to resolve a \emph{single} coherent sheaf by a locally free sheaf. This is known for schemes but is nontrivial for algebraic spaces. This yields the main theorem and we conclude by showing one may remove the reducedness hypothesis when working over a field of positive characteristic. 
\\\\
\noindent \textbf{Acknowledgments} We would like to thank Jarod Alper, Jack Hall, Laurent Moret-Bailly, David Rydh, and Wilberd van der Kallen for useful discussions and comments. This research was done in the framework of the research training group GRK 2240: Algebro-geometric Methods in Algebra, Arithmetic and Topology, which is funded by the Deutsche Forschungsgemeinschaft. The first author was also supported by the Swedish Research Council under grant no. 2016-06596 while he was in residence at Institut Mittag-Leffler in Djursholm, Sweden during the Fall of 2021.
\\\\
\noindent \textbf{Notation and conventions} We briefly recall the definition of an \emph{algebraic space} $X$ and some associated notions: points, local rings, and the topological space associated to $X$.  A functor $X: \Sch_{S}^{\text{op}} \to \Set$ is called an \emph{algebraic space} if $X$ is a sheaf on the big fppf site $S_{\text{fppf}}$, the diagonal $\Delta_S: X \to X \times_S X$ is representable by schemes, and there is a surjective \'etale morphism $U \to X$ where $U$ is a scheme.

A \emph{point} of $X$ is an equivalence class of morphisms from fields $x: \Spec K \to X$ where two morphisms are considered equivalent if there is a third dominating both. The points form a set $|X|$. Moreover, there exists a unique topology on the sets of points of algebraic spaces so that it coincides with the underlying topology of $X$ if $X$ is a scheme, morphisms $X \to Y$ induce continuous morphisms $|X| \to |Y|$, and a \'etale morphism from a scheme $U \to X$ induces a continuous open map, see \cite[Tag 03BT]{stacks}. If we suppose that $X$ is quasi-separated, then $|X|$ naturally corresponds to the set of monomorphisms $x: \Spec k \to X$ and the topology is well-behaved (see \cite[Tags 03IG, 06NI]{stacks}). For instance, if $X$ is quasi-separated every irreducible closed subset of $|X|$ has a unique generic point (see \cite[Tag 06NJ]{stacks}). \emph{All algebraic spaces considered in this paper will be quasi-separated.}

Following \cite[Tag 04KG]{stacks}, every point $x \in |X|$ can be represented by a choice of geometric point $\bar{x}: \Spec \overline{k} \to X$ and taking the colimit over all \'etale neighborhoods of $\bar{x}$ yields a strictly henselian local ring $\mathcal{O}_{X,\bar{x}}$ and a morphism $\iota_{\bar{x}}: \Spec (\mathcal{O}_{X,\bar{x}}) \to X$ which we call the \emph{strict henselization} of $x$ at $X$. Note that the isomorphism type of $\Spec (\mathcal{O}_{X,\bar{x}})$ over $X$ does not depend on the choice of geometric point $\bar{x}$. Lastly, if $\sh{F}$ is a quasi-coherent sheaf on an algebraic space $X$, and $x \in |X|$ is a point, then we define the associated stalk to be $\iota_{\bar{x}}^*(\sh{F})=\sh{F}_{\bar{x}}$ for some geometric point $\bar{x}$ lying in $x$. We will often write $x \in X$ to refer to a point $x \in |X|$.

 For an integer $n \geq 0$, a quasi-coherent sheaf $\sh{F}$ on an algebraic space $X$ satisfies Serre's condition $S_n$ if for every $x \in X$, we have 
 \[\text{depth}_{\mathcal{O}_{X,\bar{x}}}(\mathcal{F}_{\bar{x}}) \geq \text{min}(n,\text{dim}(\text{Supp}(\mathcal{F}_{\bar{x}})))\]
and an algebraic space is $S_n$ if $\mathcal{O}_X$ is. Lastly, when we write locally free we always mean locally free of finite rank.
%Note that the isomorphism type of $\Spec (\mathcal{O}_{X,\bar{x}})$ as $X$-schemes does not depend on the choice of geometric point $\bar{x}$ (see \cite[Tag 04K4]{stacks}). 
 
%Note that if $X$ is quasi-separated, then every point of $|X|$ uniquely corresponds to a monomorphism from a field $\Spec k \to X$ (see \cite[Tag 03K4]{stacks}), thus 

\section{Common affine open neighborhoods} \label{sec:common}

The goal of this section is to show that certain finite collections of points admit common affine open neighborhoods. In \cite[Tag 0ANN]{stacks} it is shown that this holds true for points with codimension $\leq 1$ but we require a stronger version that allows us to add in a further finite collections of points which themselves already admit an affine open neighborhood. This was achieved in the presence of noetherian ground rings in the proof of \cite[Thm. 1.5]{SurfaceRes} using a version of Chow's lemma. However, since we wish to remove this restriction, we require a new argument.

The following lemma from elementary topology will be used throughout the paper.

\begin{lemma} \label{open set}
Let $X$ be a topological space, $U_1\subset X$ an open set, $D=X\smallsetminus U_1$ the complementary closed set,
and $D'\subset D$ open.  Then $U_1\cup D'\subset X$ is an open set in $X$.
\end{lemma}

\proof
Write  $D'=D\cap V$ for some   some open set $V\subset X$. We have to check that $X\smallsetminus (U_1\cup D')$ is closed.
The latter  equals
$$
(X\smallsetminus U_1)\cap \left((X\smallsetminus D)\cup (X\smallsetminus V)\right) = (D\cap(X\smallsetminus D)) \cup (D\cap (X\smallsetminus V)),
$$
which is the intersection of the closed sets $D$ and $X\smallsetminus V$.
\qed

Recall that if $X$ is a noetherian scheme then there exists a normalization morphism $X' \to X$ by \cite[Tag 035N]{stacks} which is integral and that $X'$ can be viewed as the normalization of the reduction $X_{\text{red}}$ by \cite[Tag 035O]{stacks}.

\begin{Theorem} \label{prop:CKcodim1new}
Let $X$ be a scheme that is separated and noetherian, and assume that the normalization map $X'\ra X$ remains noetherian. Suppose $a_1,\ldots,a_r\in X$ are points that admit a common affine open neighborhood, and $\zeta_1,\ldots,\zeta_s\in X$ be  points of codimension at most one. Then $a_1,\ldots,a_r,\zeta_1,\ldots,\zeta_s\in X$ admit a common affine open neighborhood.
\end{Theorem}

Before giving the proof, we explain some of the necessary background. The following definition is due to Ferrand (see \cite[Thm. 5.4]{MR2044495}).

\begin{definition} \label{def:AF} A scheme $X$ is said to have the \emph{AF-property} if every finite set of points
$a_1,\ldots,a_r\in X$ admits a common affine open neighborhood. \end{definition}

\begin{Example} \label{ex:AF} The AF-property holds if $X=\text{Proj}(A)$ is the homogeneous spectrum of a graded ring, but fails in general, as observed by Nagata \cite{MR97406}. The simplest separated examples seem to be the normal surfaces $X$ constructed in \cite{MR1726231}. These are birational to the product of an elliptic curve and the projective line and have $\Sing(X)=\{a_1,a_2\}$. Note that a normal scheme $X$ that is separated and of finite type over a ground field $k$ has the AF-property if and only if it admits an ample sheaf (see \cite[Cor. 2]{MR3096912}). On the other hand, if $X$ has the AF-property and $f:X\ra Y$ is a finite birational morphism, then $Y$ also 
has the AF-property (see \cite[Thm. 5.4]{MR2044495} or \cite[Thm. 3.3]{GrossRes} for a generalization), but often fails to carry an ample sheaf (see e.g. \cite[Cor. 1]{MR284437}).\end{Example}

Some preliminary observations are in order: if $U_\lambda\subset X$, $\lambda\in L$ is a filtered inverse system of affine open sets in a scheme $X$, then the  inverse limit $U_\infty=\invlim U_\lambda$ exists as a scheme, (see \cite[Sec. 8.2]{MR217086}). It is affine, and the coordinate ring of $U_\infty$ is the direct limit of the coordinate rings for $U_\lambda$. Moreover, the formation commutes with passing to the underlying topological space. 
The canonical map $\iota:U_\infty\ra X$ is injective, its set-theoretical image is the intersection $\bigcap U_\lambda$, and the Zariski topology on the domain coincides with the subspace topology on the image.
Moreover, for all image points $\iota(x)\in X$ the canonical map $\mathcal{O}_{X,\iota(x)}\ra \mathcal{O}_{U_\infty,x}$ is bijective. By abuse of notation, we also write the monomorphism $\iota$ as an inclusion  $U_\infty\subset X$, and call it an \emph{affine pro-open set}.
 
Now let $a_1,\ldots,a_r\in X$ be points admitting a common affine open neighborhood. Let $U_\lambda$, $\lambda \in L$ be the ordered set of all such neighborhoods. We call $U_\infty\subset X$ the \emph{minimal affine pro-neighborhood} of the $a_1,\ldots,a_r \in X$. It comprises all points $x\in X$ that specializes to one of the $a_i$.
One may compute it explicitly as follows: Fix an index $\lambda\in L$ and write  $U=U_\lambda$. Let $R=\Gamma(U,\mathcal{O}_X)$ be the coordinate ring, and $\primid_i\subset R$ be the prime ideals corresponding to the $a_i\in U$. With the multiplicative system $S=R\smallsetminus (\primid_1\cup\ldots\cup\primid_r)$ we get a canonical identification $U_\infty=\Spec(S^{-1}R)$. If $X$ is integral one furthermore has 

$$
S^{-1}R=R_{\primid_1}\cap\ldots\cap R_{\primid_r} =\mathcal{O}_{X,a_1}\cap \ldots\cap \mathcal{O}_{X,a_r},
$$

\noindent where the intersection takes place in the function field $F=\mathcal{O}_{X,\eta}$ of our scheme (\cite[Ch. 2, Sec. 3, No. 5 Prop. 17]{MR979760}).

\proof
In light of \cite[Cor. 4.5.9]{EGAI.2}, it suffices to treat the case that $X$ is reduced.
We proceed by induction on $s\geq 0$. The case $s=0$ is trivial. Suppose now that 
$s\geq 1$, and that the assertion is true for $s-1$. Setting $a_{r+j}=\zeta_j$ for $1\leq j\leq s-1$
and replacing $r$ by $r+s-1$, 
we are reduced to the case $s=1$. To simplify notation   put $\zeta=\zeta_1$.

We first treat the case that the scheme $X$ is integral and normal. The assertion is immediate if $\zeta\in X$
is the generic point. Suppose now that the local ring $\mathcal{O}_{X,\zeta}$ has dimension one.
Choose affine open sets $U,V\subset X$ that contain  the $a_1,\ldots,a_r$ and $\zeta$, respectively.
We are done if $\zeta$ belongs to $U$, so     assume that  $\zeta\not\in U$.
Without restriction, we may assume $X=U\cup V$. Write $U=\Spec(R)$,   let $\primid_i\subset R$ be the prime
ideals corresponding to the $a_i\in U$, and consider the minimal affine pro-neighborhood $U_\infty=\Spec(S^{-1}R)$ as above.
This lies in $U$, and thus does not contain $\zeta$. 

Let $F=\text{Frac}(R)=\mathcal{O}_{X,\eta}$ be the function field of the integral scheme $X$. Each point $x\in X$
yields a   subring $\mathcal{O}_{X,x}\subset F$. Recall that  $x\in \overline{\{y\}}$
if and only if $\mathcal{O}_{X,x}\subset \mathcal{O}_{X,y}$, according to \cite[Cor. 8.5.7]{EGAI.2}.
We now check that similar facts hold for semi-local rings instead of local rings.
Seeking a contradiction, assume $S^{-1}R \subset \mathcal{O}_{X,\zeta}$.
The fiber product $\Spec(S^{-1}R)\times_X\Spec(\mathcal{O}_{X,\zeta})$ is affine, and its coordinate ring $C$
is generated by $S^{-1}R$ and $\mathcal{O}_{X,\zeta}$, because $X$ is separated. Moreover, the  projections turn $\Spec(C)$
into pro-affine sets inside the factors $\Spec(S^{-1}R)$ and $\Spec(\mathcal{O}_{X,\zeta})$, and we may regard $\Spec(C)\subset X$
as a pro-affine set. It comprises the points $x\in X$ that specialize to  one of the $a_i$, and also to $\zeta$.
In particular $\zeta$ does not belong to  $\Spec(C)$. Using that $C$ is contained in  $F$, and is generated inside $F$
by the subrings  $S^{-1}R$ and $\mathcal{O}_{X,\zeta}$, we infer  $C=\mathcal{O}_{X,\zeta}$ and conclude $\zeta\in\Spec(C)$, contradiction.
Summing up,  $S^{-1}R\not\subset \mathcal{O}_{X,\zeta}$.
Choose an element $f\in S^{-1}R$ that does not belong to $\mathcal{O}_{X,\zeta}$.
After shrinking the affine open set $U$, we may assume that $f\in R$, and  regard it as a local
section $f\in \Gamma(U,\mathcal{O}_X)$. 

Recall that we are assuming that $X$ is normal.
By the Serre Criterion, the local ring $\mathcal{O}_{X,\zeta}$ is a discrete valuation ring.
Let $\val:F^\times \ra \mathbf{Z}$ be the corresponding valuation. Then $\val(f)<0$,
thus $f^{-1}\in F^\times$ belongs to the maximal ideal $\maxid_\zeta$. After shrinking the affine open set $V$, we may 
assume $f^{-1}\in\Gamma(V,\mathcal{O}_X)$.
All irreducible components of $D=X\smallsetminus U$ have codimension one (\cite[Cor. 21.12.7]{EGAIV4}). Choose an open set $D_0\subset D$
that contains $\zeta\in D$ but no other generic point of $D$. Then   $U\cup D_0\subset X$ is an open set, according to Lemma \ref{open set}.
Shrinking $V$ further, we may assume that $D=\overline{\{\zeta\}}$. In turn, the effective Cartier-divisor $V(f^{-1})\subset V$
coincides with the complement $D=X\smallsetminus U$ as closed sets.
  
By construction we have $f,f^{-1}\in\Gamma(U\cap V,\mathcal{O}_X)$. In turn, this gives a cocycle $f^{-1}\in\Gamma(U\cap V,\mathcal{O}_X^\times)$
and thus an invertible sheaf $\sh{L}$ on $X$. It has $\sh{L}|_U=\mathcal{O}_U$ and $\sh{L}|_V=\mathcal{O}_V$, and the 
global sections are pairs $s=(s_U,s_V)$ satisfying $f^{-1}s_U=s_V$ on the overlap $U\cap V$.
In particular, we have the global sections 
$$
s_0=(1,f^{-1})\quadand s_1=(f,1).
$$
Consequently, $\sh{L}$ is globally generated,
and thus defines a morphism 
$h:X\ra P(X,\sh{L})$,
taking values in the homogeneous spectrum $P=P(X,\sh{L})$ of the graded ring $R(X,\sh{L})=\bigoplus_{t\geq 0} H^0(X,\sh{L}^{\otimes t})$.
The quasi-coherent sheaf $\mathcal{O}_P(1)$ is invertible on the union of all $D_+(s)\subset P$,  $s\in\Gamma(X,\sh{L})$. The morphism $h:X\ra P$ 
factors over this union, and we have $\sh{L}=h^*(\mathcal{O}_P(1))$. Note that the points $h(a_1),...,h(a_r), h(\zeta)$ all land in an open subscheme $Y \subset P$ over which $\mathcal{O}_P(1)$ is an ample invertible sheaf. Thus, there is an affine open subscheme $W \subset P$ which contains these points and where $\mathcal{O}_P(1)|_W$ is a trivial line bundle. So by replacing $P$ with $W$ and $X$ with $h^{-1}(W)$ we may suppose that $\sh{L}$ is a trivial line bundle. We will show that it is ample after perhaps further shrinking $X$.

By construction, the non-zero locus $X_{s_0}=U$ is affine. Choose a local section $\varphi\in\Gamma(U,\sh{L})$ that vanishes on the closed set $U\smallsetminus V$,
but not at the generic point $\eta\in U$. For some $n\geq 0$, the local section $\varphi\cdot s_0^{\otimes n}|U$ extends
to some global section $s_2\in \Gamma(X,\sh{L}^{\otimes 1+n})$. The corresponding effective Cartier divisor $H\subset X$
is of the form $H=mD+H'$, for some multiplicity $m\geq 0$ and $D\not\subset H'$. Likewise, our section $s_0\in \Gamma(X,\sh{L})$
yields an effective Cartier divisor $H_0=dD$ for some $d\geq 1$. Thus $s=s_2^{\otimes d}\otimes s_0^{\otimes -m}\in \Gamma(X,\sh{L}^{dn+d-m})$
vanishes on $X\smallsetminus V$ but not at $\zeta$. Thus the non-zero locus $X_s$ contains $\zeta$, lies in $V$,
and there takes the form $V'=D(s|_V)$, whence is affine. It follows that $\mathcal{L} \cong \mathcal{O}_X$ is ample. Therefore $X$ is quasi-affine by \cite[Tag 01QE]{stacks} and so it must have the AF-property.

This settles the case that $X$ is integral and normal. It remains to treat the general case where $X$ is reduced.
Let $\nu:X'\ra X$ be the normalization map, which is integral with set-theoretically finite fibers by \cite[Prop. 4.8.2]{MR2266432}, in particular closed and affine. 
The fiber $\nu^{-1}(\zeta)$ consists of finitely many codimension-one points,
and $ \nu^{-1}(\{a_1,\ldots, a_r\})$ is a finite set contained in some affine open set.
We saw above that there is an affine open set $U'\subset X'$ containing $\nu^{-1}(\zeta)$ and the $\nu^{-1}(a_i)$.
Consider the complementary closed set $Z'=X'\smallsetminus U'$.
The image $\nu(Z')\subset X$ is closed and disjoint from $a_1,\ldots,a_r$ and $\zeta$.
Thus 
$$
\nu^{-1}(X\smallsetminus \nu(Z')) = X'\smallsetminus \nu^{-1}(\nu(Z')) \subset X'\smallsetminus Z'=U'
$$
is quasiaffine. 
Replacing $X$ by the complement of $\nu(Z')$, we thus may assume that $Z'$ is empty and $X'$ is quasiaffine.
In particular, $X'$ has the AF-property and by \cite[Thm. 3.3]{SurfaceRes} the scheme $X$ has the AF-property.
In particular, the points $a_1,\ldots,a_r,\zeta\in X$ admit a common affine open neighborhood. \qed

\medskip

The preceding result on schemes can be improved using the fact that codimension $\leq 1$ points on algebraic spaces are schematic. 

\begin{corollary} \label{cor:codim1}  Let $X$ denote a separated noetherian algebraic space whose normalization $X'$ remains noetherian. Suppose $\zeta_1,...,\zeta_n, a_1,...,a_r \in X$ are as in the statement of Theorem \ref{prop:CKcodim1new}. Then there is an open affine subscheme $U \subset X$ which contains all of $a_1,..,a_r,\zeta_1,...,\zeta_n$.
\end{corollary}

\begin{proof} By \cite[Tag 0ADD]{stacks}, there is an open subscheme $U \subset X$ containing all points $x \in |X|$ whose strictly local rings have $\text{dim} \mathcal{O}^{\text{sh}}_{X,x} \leq 1$. Thus, all the points in the statement belong to the largest open subscheme $W \subset X$. Thus, we may suppose that $X$ is a scheme and to conclude we apply Theorem \ref{prop:CKcodim1new}. \end{proof}

\begin{remark} The previous results are sharp in the sense that there exists a smooth proper $k$-scheme $X$ and two points $x,\zeta \in X$ where $x$ is closed and $\zeta$ is codimension $2$ for which there is no quasiprojective open subset $U \subset X$ containing both $x$ and $\zeta$. As an example consider Hironaka's well-known construction of a smooth non-projective threefold which has two rational curves $C,D \subset X$ whose union is algebraically equivalent to $0$. In this situation, let $x \in C$ and $\zeta$ be the generic point of $D$. If $U \subset X$ was a quasi-projective open subset containing both points, then there would be an affine open subset $V \subset U$ containing both points. However, the divisor $H=X \smallsetminus V$ would intersect $C \cup D$ nontrivially, contradicting the fact that this reducible curve is algebraically trivial. Later we will show that we may find a desirable open neighborhood of two such points $U \subset X$ if we relax the quasi-projectivity requirement (see Proposition \ref{prop:affineMV}). \end{remark}

\section{Flat neighborhoods} \label{Flat neighborhoods} 
In general, algebraic spaces are not covered by affine open subschemes. However, given any finite collection of codimension $\leq 2$ points $\zeta_1,...,\zeta_n$, we will construct an open neighborhood around the $\zeta_i$ which satisfies the resolution property. We begin with some definitions.

\begin{definition} \label{def:flatneighborhood} Let $C \subset X$ be a closed subspace of an algebraic space, a \emph{flat neighborhood} of $C$ is a flat morphism $U \to X$ with the following property: for any test scheme $T$ and any morphism $T \to X$ with $T \times_X (X \smallsetminus C) = \varnothing$, the first projection $T \times_X U \to T$ is an isomorphism. \end{definition}

\begin{remark} \label{rem:flatnotherian} Suppose $C \subset X$ is a closed subspace defined by the ideal sheaf $\sh{I}_C$ and $C_i$ the closed subspace defined by $\sh{I}_C^{i+1}$ for $i \geq 1$, then for any flat neighborhood $U \to X$ of $C$, the first projection $C_i \times_X U \to C_i$ is an isomorphism. In fact, if $X$ is locally noetherian the converse holds: if $U \to X$ is a flat morphism and for every $i \geq 1$ the map $C_i \times_X U \to C_i$ is an isomorphism, then $U \to X$ is a flat neighborhood of $C$ (see \cite[Lem 3.2, Lem. 3.3]{hall2016mayer}). \end{remark}

We will study fpqc coverings of the form $(U_i \to X)_{i=1,2}$ where $U_1 \to X$ is a quasi-compact open immersion and $U_2 \to X$ is a flat neighborhood of the closed subspace $C=X \smallsetminus U_1$ endowed with the reduced structure. Using the notation of \cite{hall2016mayer}, the resulting cartesian squares are called \emph{flat Mayer --Vietoris} squares:

\[\begin{tikzcd}
  U_1 \times_{X} U_2 \arrow[d] \arrow[r] & U_2 \arrow[d]  \\
  U_1 \arrow[r] &  X
  \end{tikzcd}
\]
\\\\
Now we explain how to construct affine flat neighborhoods. Let $X$ be a noetherian algebraic space with \emph{affine diagonal}, and $Z \subset X$ be an affine closed subspace, corresponding to some coherent sheaf of ideals $\sh{I} \subset \mc{O}_X$. Then the powers $\sh{I}^{i+1}$ define infinitesimal neighborhoods $Z^{(i)} \subset X$, which form a direct system
of closed subspaces. One may regard the ind-scheme as a formal neighborhood
of $Z\subset X$.

According to \cite[Tag 05YU]{stacks} the infinitesimal neighborhoods are affine as well, and the direct system $Z^{(i)}$ corresponds to an inverse system of noetherian rings $A_i$ with surjective transition maps.
Form the inverse limit $A=\invlim A_i$ and write $U=\Spec(A)$ for its spectrum. The induced projections $A \to A_i$ are surjective, so the canonical morphisms  $Z^{(i)} \to U$ are closed embeddings.

\begin{proposition} \label{prop:flatnhbd} In the above situation, there is a unique morphism $f:U \to X$ making the diagrams
%Let $X$ denote a noetherian algebraic space and $Z \subset X$ an affine closed subscheme. There is a unique morphism from an affine scheme $f:U \to X$ making the diagrams
$$
\begin{tikzcd} 
						& U\ar[d,"f"]\\	 
Z^{(i)}\ar[r]\ar[ur]	& X
\end{tikzcd}
$$
commutative. Moreover, the morphism $f$ is a flat neighborhood of $Z$, and the set-theoretic image $f(U)\subset X$ consists of all points $\zeta\in X$ whose closure $\overline{\{\zeta\}}$ intersects $Z$.  
\end{proposition}

\begin{proof} The existence of a map $f: U \to X$ forming a flat neighborhood of $Z$ is proven in \cite[Thm. 34]{SidRes} and uniqueness follows from \cite[Cor. 1.5]{hall2014coherent}.

Note that $Z$ lies in the image of $f$ and so if the closure of a point $\zeta \in X$ intersects $Z$, $\zeta$ is a generization of a point in $f(U)$. Since the image of a flat morphism is stable under generization $f(U)$ must contain all such $\zeta$. On the other hand, if the closure of a point $\zeta \in X$ misses $Z$ then there is an open neighborhood $W \subset X$ of $Z$ that is disjoint from $\overline{\{\zeta\}}$. However, the argument showing the existence of $f: U \to X$ shows that it factors through $W \subset X$ since each $Z^{(i)}$ does as well. This shows $f(U)$ consists of those points $\zeta \in X$ whose closure intersects $Z$. \end{proof}

\medskip
Now suppose that $\sh{I}' \subset \sh{I}$ is another sheaf of ideals with the same radical.
This gives another inverse limit $A'=\invlim A'_i$, together with a flat neighborhood $U'=\Spec(A') \to X$ of $Z'=V(\sh{I}') \subset X$. 
Moreover, we have a homomorphism $A' \to A$ of rings and a map $g:U \to U'$ which is a morphism of $X$-schemes by the uniqueness in Proposition \ref{prop:flatnhbd}. 

\begin{proposition}
The natural morphism $g: U \to U'$ is an isomorphism. 
\end{proposition}

\begin{proof} There are positive integers $n$ and $m$ with $\sh{I}^n \subset \sh{I'} \subset \sh{I}$ and $\sh{I'}^m \subset \sh{I}^n \subset \sh{I}'$. Thus, if $U_n$ and $U'_m$ denote the flat neighborhoods of $Z^{(n)}$ and $Z'^{(m)}$, we obtain maps
\[U \to U' \to U_n\]
\[U' \to U_n \to U_m'\]
and if we show both compositions are isomorphisms, it will follow that $U \to U'$ is an isomorphism. Thus, it suffices to assume that $\sh{I}'=\sh{I}^n$ by symmetry. Now, the result follows from the fact that the natural map $\invlim_n (B/I^{n}) \to \invlim_m (B/I^{nm})$ is an isomorphism for any commutative ring $B$ and any ideal $I \subset B$.\end{proof}

\medskip
It follows that the flat neighborhood $f:U=\Spec A \to X$ only depends
on the closed set $|Z| \subset |X|$, or equivalently on the complementary open subspace $X \smallsetminus Z$,
and not on the chosen scheme structure of $Z$. To conclude, we show that there are plenty of open neighborhoods $U \subset X$ which appear in a flat Mayer--Vietoris square with affine schemes. 

\begin{corollary} \label{cor:mvsquaresexist} Let $X$ be a noetherian algebraic space with affine diagonal, if $V=\Spec A \subset X$ is an affine open subset such that the reduced closed subset $C=(X \smallsetminus V)_{\mathrm{red}}=\Spec B \subset X$ is an affine scheme, then $X$ appears in a flat Mayer--Vietoris square with affine schemes:
\[\begin{tikzcd}
  C \times_X V \arrow[d] \arrow[r] & U=\Spec B' \arrow[d, "f"]  \\
  V=\Spec A \arrow[r, "i"] &  X
  \end{tikzcd}
\]
\end{corollary}

\begin{proof} Proposition \ref{prop:flatnhbd} applied to the closed immersion $C \to X$ yields an affine flat neighborhood $U \to X$ of $C$. Thus, by taking the fiber product of $U$ and $V$ over $X$ yields a flat Mayer--Vietoris with affine schemes, as desired. \end{proof}

We conclude this section by using flat neighborhoods to improve Corollary \ref{cor:codim1}. 

\begin{proposition} \label{prop:affineMV} Let $X$ denote a separated noetherian algebraic space whose normalization remains noetherian and suppose that $z_1,...,z_n \in X$ all have codimension $\leq 2$ and that $x_1,...,x_r$ all admit a common affine open neighborhood. Then there is an open neighborhood $U \subset X$ of the $z_1,...,z_n,x_1,...,x_r$ that appears in a flat Mayer--Vietoris square with affine schemes. \end{proposition}

\begin{proof} We will assume that the set $\{z_1,...,z_n\}$ contains all the generic points of $X$. Suppose $z_1,...,z_m \in \{z_1,...,z_n\}$ denotes the points with codimension $\leq 1$. Then, by Corollary \ref{cor:codim1} there is an affine open subscheme $\Spec A \subset X$ containing the $z_1,...,z_m,x_1,..,x_r$. If $D$ denotes $X \smallsetminus \Spec A$ with the reduced structure, the points $z_{m+1},...,z_n$ which are not already in $\Spec A$ now have strictly local rings $\mathcal{O}^{\text{sh}}_{D,z_i}$ of dimension $\leq 1$ because $\Spec A \subset X$ is dense. By Proposition $\ref{cor:codim1}$ again, there is an affine open $\Spec B \subset D$ containing the $z_{m+1},...,z_n$ which are not already in $\Spec A$. Consider the open set $W=\Spec B \cup \Spec A \subset X$ and note that it satisfies $|W| \cap |D|=|\Spec B|$. Then the open neighborhood $W$ appears in a flat Mayer--Vietoris square with affine schemes. \end{proof}

\begin{remark} \label{rem:strats} The reason the above argument works is that codimension $\leq 1$ points admit affine neighborhoods on algebraic spaces. So even if a codimension $2$ point of $X$ isn't scheme-like, it becomes so after viewing the point on the complement of the dense affine open neighborhood $\Spec A \subset X$. The same argument above would work for non-schematic points $x \in X$ if we could find a desirable dense open $\Spec A \subset X$ so that $x \in X \smallsetminus A$ admits an affine open neighborhood. 

We do not know if this holds for algebraic spaces in general, however there is one interesting case where we can say something: separated and finite type algebraic spaces over $\widebar{\mathbb{F}}_p$ of dimension three. Indeed, if $X$ is such an algebraic space and $\Spec A \subset X$ is any dense open subset, then the reduced complement $Y$ is a scheme because the normalization $\tilde{Y} \to Y$ yields a quasi-projective scheme $\tilde{Y}$ by \cite[Cor. 2.11]{MR146182} and by \cite[Thm. 3.3]{SurfaceRes} $Y$ is an AF-scheme.

In general all we can say is that $X$ admits a stratification of open immersions:
\[\varnothing=U_{0} \subset U_1 \subset \dots \subset U_n=X\]
where each $U_{i+1} \smallsetminus U_{i}$ is an affine scheme. This follows by repeatedly applying Corollary \ref{cor:codim1}. \end{remark} 

This motivates the following question

\begin{Question} Let $x$ denote a point on a separated and noetherian algebraic space $X$ which is not scheme-like. Does there exist a dense affine open set $U \subset X$ such that $x \in X \smallsetminus U$ admits an affine open neighborhood? More precisely, what is the least integer $p$ so that there is a stratification as in Remark \ref{rem:strats} where $x \in U_p$? \end{Question}

A positive answer would help us to extend Theorem \ref{resolutionthree} to points which are not scheme-like.
\section{Triples}
\label{Triples}

In this section we study a very concrete and simple situation which lies entirely in the realm of commutative algebra. Suppose that we have a diagram of rings
\begin{equation}
\label{diagram rings}
\begin{CD}
R_{12}	@<<<	R_2\\
@AAA\\
R_1.
\end{CD}
\end{equation}
Write $\Mod(R_i)$ for the abelian categories of $R_i$-modules, and consider the resulting  2-fiber product  
$$
\catA=\Mod(R_1)\times_{\Mod(R_{12})} \Mod(R_2).
$$
The objects  in this category are \emph{triples} $(M_1,M_2,\varphi)$ where $M_1$ is an $R_1$-module, $M_2$ is an $R_2$-module,
and $\varphi:M_1\otimes R_{12}\ra M_2\otimes R_{12}$ is an isomorphism.
The homomorphisms $(M_1,M_2,\varphi)\ra (N_1,N_2,\psi)$ are pairs $(f_1,f_2)$ of homomorphisms 
$f_i:M_i\ra N_i$ such that $\psi\circ (f_1\otimes\id) = (f_2\otimes\id)\circ\varphi$.
Note that $(f_1,f_2)$ is an isomorphism in $\catA$ if and only if the maps $f_1$ and $f_2$ are bijective. 

Also note  that one may complete the diagram \eqref{diagram rings} with the ring $R=R_1\times_{R_{12}} R_2$ to a cartesian square,
and study the resulting functor 
$$
\text{Mod}(R) \lra \sh{A},\quad M\longmapsto (M\otimes_R R_1, M\otimes_R R_2,\can).
$$
This was done, under suitable assumptions,  by Milnor \cite[Sec. 2]{MR0349811} or Ferrand \cite[Sec. 5]{MR2044495}. However, we shall see below that often one has to replace the fiber product ring $R$ by a suitable scheme or algebraic space $X$ in order to uncover the hidden geometry of the situation. For the time being, we make a purely algebraic study of the diagram \eqref{diagram rings}. We start with the following observation:

\begin{proposition}
\label{abelian category}
The category $\sh{A}$ is additive. It is an abelian category provided the 
 homomorphisms $R_1\ra R_{12}$ and $R_2\ra R_{12}$ are flat.  
\end{proposition}

\proof
Obviously, addition of the maps $f_i$ endows the Hom sets with a group structure for which composition is bilinear.
Moreover $(0,0,\id)$ is an initial object, and 
$$
(M_1,M_2,\varphi)\times (N_1,N_2,\psi) = (M_1\oplus N_1, M_2\oplus N_2, \varphi\oplus\psi)
$$
is a product. Thus $\sh{A}$ is an additive category (see \cite[Lem. 8.2.9]{MR2182076}).
Given a  morphism $f=(f_1,f_2):(M_1,M_2,\varphi)\ra (N_1,N_2,\psi)$ one sets $C_i=\text{Cokernel}(f_i)$
and writes $\beta:C_1\otimes R_{12}\ra C_2\otimes R_{12}$ for the induced map. The latter is bijective, by the Five Lemma,
and one can check that $(C_1,C_2,\beta)$ is a cokernel for $(f_1,f_2)$ in the category $\sh{A}$. 

To construct kernels we assume that  $R_i\ra R_{12}$ are flat. Set $K_i=\Kernel(f_i)$.
Then $K_i\otimes_{R_i}R_{12}$ is the kernel for $M_i\otimes_{R_i}R_{12}\ra N_i\otimes_{R_i}R_{12}$,
and we write $\alpha:K_1\otimes_{R_1}R_{12}\ra K_2\otimes_{R_2}R_{12}$ for the induced map.
Again one can check that $(K_1,K_2,\alpha)$ is a kernel for $(f_1,f_2)$. 

Lastly, we need to check that the natural map from the cokernel of $(K_1,K_2, \alpha) \to (M_1,M_2,\varphi)$
to the kernel of $(M_1,M_2,\varphi) \to (C_1,C_2, \beta)$ is an isomorphism. However, this is true because it holds on each component. \qed

\medskip
In our setting, objects of the form $(R_1^{\oplus n},R_2^{\oplus n}, A)$ for some $n\geq 0$ and some $A\in \GL_n(R_{12})$
take over the role of locally free sheaves of constant rank which are trivial over $R_1$ and $R_2$. In light of this, the following is relevant for the resolution property.

\begin{Theorem}
\label{resolution property from commutative algebra}
Suppose that $R_i\ra R_{12}$   are flat for $i=1,2$ and that the multiplication map $R_{12}\otimes_{R_2}R_{12}\ra R_{12}$
is bijective.
Then every object $(M_1,M_2,\varphi)$ with finitely generated $M_1$ and $M_2$ 
is the quotient of some $(R_1^{\oplus n}, R_2^{\oplus n}, A)$  for a suitable $n\geq 0$ and $A\in \GL_n(R_{12})$.
\end{Theorem}

\proof
We first check that $(M_1,M_2,\varphi)$ is the quotient of some $(R_1^{\oplus d},N_2,\psi)$ with some integer $d\geq 0$
and some finitely generated $N_2$.
Since $M_1$ is finitely generated, there is indeed a surjection $R_1^{\oplus d}\ra M_1$.
Write $h:R_{12}^{\oplus d}\ra M_2\otimes_{R_2} R_{12}$ for the composition of the induced map
$R_{12}^{\oplus d}\ra M_1\otimes_{R_1}R_{12}$ with $\varphi$.
The cartesian diagram
$$
\begin{CD}
R_{12}^{\oplus d} 	@>h >>	M_2\otimes_{R_1}R_{12}\\
@AAA				@AA\can A\\
F	@>>>	M_2
\end{CD}
$$
defines some $R_2$-module $F$. The lower horizontal arrow is surjective, because this holds for the upper horizontal arrow.
Tensoring this diagram with  $R_{12}$ yields another diagram
$$
\begin{CD}
R_{12}^{\oplus d}\otimes_{R_2}R_{12} 	@>>>	M_2\otimes_{R_1}R_{12}\otimes_{R_2}R_{12}\\
@AAA					@AA\can A\\
F\otimes_{R_2}R_{12}			@>>>	M_2\otimes_{R_2}R_{12},
\end{CD}
$$
which remains cartesian because $R_{12}$ is flat over $R_1$. The vertical arrow on the right is bijective,
because $R_{12}\otimes_{R_2}R_{12}\ra R_{12}$ is bijective.
Hence the vertical arrow to the left is bijective as well. In a similar way we get an identification
$$
R_{12}^{\oplus d}\otimes_{R_2}R_{12}= R_{12}^{\oplus d}\otimes_{R_{12}}R_{12}\otimes_{R_2}R_{12} =R_{12}^{\oplus d}.
$$
Let $F_\lambda\subset F$, $\lambda\in L$ be the family of finitely generated submodules.
For some sufficiently large index $\mu$, the composite map  $F_\mu\ra M_2$ is surjective and   
the inclusion $F_\mu\otimes_{R_2}R_{12}\subset R_{12}^{\oplus d}$ is  bijective. This can be seen as  an isomorphism $\psi:R_{12}^{\oplus d}\ra F_\mu\otimes_{R_2}R_{12}$. Setting $N_2=F_\lambda$ we obtain the desired object $(R_1^{\oplus d},N_2,\psi)$.

To see that $(M_1,M_2,\varphi)$ is the quotient of some $(R_1^{\oplus n}, R_2^{\oplus n},A)$ it thus suffices to treat the case  $M_1=R_1^{\oplus d}$.
Choose a surjection $f_2:R_2^{\oplus n}\ra M_2$ for some $n\geq 0$. Since $M_2\otimes_{R_2}R_{12}$ is free of rank $d$, the kernel $K$ for the induced surjection
$R_{12}^{\oplus n}\ra M_2\otimes_{R_2}R_{12}$ is locally free of rank $n-d$, and in particular $n\geq d$.
Hence $K\oplus R_{12}^{\oplus m}$ is free, for some integer $m\geq 0$. Replacing $f_2$ by the surjection $(f_2,0):R_2^{\oplus n}\oplus R_2^{\oplus m}\ra M_2$,
we thus may assume that $K$ is free. Choose an isomorphism $\alpha: R_{12}^{\oplus n-d}\ra K$.
The surjection $f_2 \otimes \id \colon R_{12}^{\oplus n}\ra M_2\otimes_{R_2}R_{12}$ admits a section $\sigma$, because
the range is free. This yields a commutative diagram
$$
\begin{CD}
R_{12}^{\oplus d}\oplus R_{12}^{\oplus n-d}	@>(\sigma\circ\varphi,\alpha)>>	 R_{12}^{\oplus n}\\
@V(\id,0)VV					@VVf_2\otimes\id V\\
R_{12}^{\oplus d}			@>>\varphi>	M_2\otimes_{R_2}R_{12}
\end{CD}
$$
with bijective horizontal arrows. Let $A\in\GL_n(R_{12})$ be the matrix corresponding to the upper horizontal arrow,
and let $f_1:R_1^{\oplus n}\ra R_1^{\oplus d}=M_1$ be the projection onto the first $d$ factors.
It follows that $(f_1,f_2)$ defines the desired surjection from  $(R_1^{\oplus n}, R_2^{\oplus n}, A)$   onto $(M_1,M_2,\varphi)$.
\qed
\\\\
The condition $R_{12}\otimes_{R_1}R_{12}=R_{12}$ in Theorem \ref{resolution property from commutative algebra} may seem strange at first glance but, as the following proposition shows, it is quite natural. 

\begin{proposition} \label{lem:openimmersion}
Let $\varphi:R\ra A$ be a homomorphism of rings. Then the following are equivalent:
\begin{enumerate}
\item There are elements $f_1,\ldots, f_r\in R$  such that the induced maps on localizations $R_{f_i} \ra A_{f_i}$  are bijective for $1\leq i\leq r$, and that $\sum_{i=1}^r Af_i = A$. 
\item The induced morphism $u: \Spec(A)\ra\Spec(R)$ is an open immersion.
\item
The $R$-algebra $A$ is flat, of finite presentation, and the multiplication map $A\otimes_R A \ra A$ is bijective.
\end{enumerate}
\end{proposition}

\proof The equivalence $(1) \Leftrightarrow (2)$ is \cite[Prop. 2]{MR286805}. The statement $(2) \Rightarrow (3)$ is straightforward so we omit it. We will show that $(3) \Rightarrow (2)$, i.e. that if $u$ is a fppf morphism such that $A \otimes_R A \to A$ is a bijection, then $u$ is an open immersion. Since the multiplication map is bijective, it induces a section of the map $\text{pr}_1: \Spec (A \otimes_R A) \to \Spec (A)$ which is an isomorphism. Since the set-theoretic image $u(\Spec A)=U$ is open in $\Spec(R)$ (see \cite[Tag 01UA]{stacks}), it suffices to show that $u|_{u^{-1}(U)}$ is an open immersion but this is an open immersion fppf-locally and the result now follows from \cite[Tag 02L3]{stacks}.
\qed

\begin{corollary} \label{MVresolution} Let $X$ denote a quasi-compact and quasi-separated algebraic space which appears in a flat Mayer--Vietoris square of the following form:
\[\begin{tikzcd}
  \Spec D \arrow[d] \arrow[r] & \hat{Z}=\Spec \hat{C} \arrow[d, "f"]  \\
  U=\Spec A \arrow[r, "i"] &  X
  \end{tikzcd}
\]
where $i$ is an open immersion and $f$ is a flat neighborhood of $X \smallsetminus U$. Then $X$ satisfies the resolution property.
\end{corollary}

\begin{proof} By \cite[Thm. B (1)]{hall2016mayer}, the natural functor
\[\Qcoh(X) \to \Qcoh(\Spec A) \times_{\Qcoh(\Spec D)} \Qcoh(\Spec \hat{C})\]
is an equivalence of categories. Moreover, the ring maps $D \to \hat{C}$ and $D \to A$ are both flat, the top horizontal arrow satisfies the equivalent properties of Proposition \ref{lem:openimmersion} and so we may apply Theorem \ref{resolution property from commutative algebra}. In particular, this shows that every quasi-coherent sheaf of finite type in $X$ is the quotient of a locally free sheaf on $X$ which is trivial when restricted to $U$ and $\hat{Z}$. \end{proof} 

\section{Locally free resolutions} \label{resolutions}

In this section, we show that given a coherent sheaf $\sh{F}$ on a separated algebraic space $X$, it can be resolved by a locally free sheaf away from a closed subset of codimension $\geq 3$, as long as $X$ satisfies some mild conditions. Crucially, this closed subset depends on the sheaf $\sh{F}$. We begin by recording some of these technical conditions.
\begin{assumption} \label{conditions} Let $X$ be a noetherian algebraic space satisfying the following conditions. 
\begin{enumerate} 
\item Every integral closed subscheme $C \subset X$ of codimension $\leq 0$ contains a dense open regular subset, and 
\item the normalization $\tilde{X}$ of $X$ is a noetherian algebraic space. \end{enumerate}
\end{assumption} 

\noindent These conditions are satisfied if, for example, $X$ is a quasi-excellent algebraic space (see \cite[Remark 7.2]{hall2014coherent}).

Next, we record a few lemmas showing that we may replace a given coherent sheaf $\sh{F}$ with one that enjoys better properties. We say a sheaf $\sh{F}$ on $X$ is \emph{supported everywhere} if $\text{Supp}(\sh{F})=X$. 

\begin{lemma} \label{S_1} Let $X$ be a noetherian algebraic space with affine diagonal which is $S_1$. Then any coherent sheaf $\sh{H}$ is the quotient of a coherent sheaf $\sh{F}$ which is $S_1$ and is supported everywhere. \end{lemma}

\begin{proof} Since $X$ is noetherian and has affine diagonal, there is a faithfully flat affine morphism $f: U \to X$ where $U$ is an affine scheme. Now \cite[Lem. 2.14]{GrossRes} implies $\sh{H}$ is the quotient of a quasi-coherent subsheaf $\sh{F} \subset f_*\mathcal{O}_{U}^{(I)}$ for some index set $I$. Since $\sh{H}$ is of finite type, $f_*\mathcal{O}_{U}^{(I)}$ is the union of its coherent subsheaves and is supported everywhere, we may choose $\sh{F}$ to be of finite type and supported everywhere as well.

To see that $\sh{F}$ is $S_1$ it suffices to show that at any geometric point $\bar{p} \in X$, there is an element of the associated strictly local ring $\mathcal{O}_{X,\bar{p}}^{\text{sh}}$ which is a nonzero divisor of $\sh{F}_{\bar{p}}=\sh{F} \otimes_{\mathcal{O}_X} \mathcal{O}_{X,\bar{p}}^{\text{sh}}$. Note that because $\mathcal{O}_{X,\bar{p}}^{\text{sh}}$ is $S_1$, it admits a non-zero divisor $r \in \mathcal{O}_{X,\bar{p}}^{\text{sh}}$ and because $f$ is faithfully flat, $r$ is not a zero divisor for $f_*\mathcal{O}_{U}^{(I)} \otimes_{\mathcal{O}_X} \mathcal{O}_{X,\bar{p}}^{\text{sh}}$ either. Since $\sh{F}_{\bar{p}}$ is a subsheaf of $f_*\mathcal{O}_{U}^{(I)} \otimes_{\mathcal{O}_X} \mathcal{O}_{X,\bar{p}}^{\text{sh}}$ it follows that $r$ is not a zero divisor for $\sh{F}_{\bar{p}}$. \end{proof}

Another lemma we will require is the following.

\begin{lemma} \label{stwocoherent} Let $X$ be a noetherian and separated algebraic space which is reduced and satisfies the assumptions in \ref{conditions}. If $\sh{F}$ is a $S_1$ coherent sheaf on $X$ which is supported everywhere, then the set
\[\mathbf{U}_{S_2}(\sh{F})=\{x \in X|\sh{F}_{\bar{x}} \textrm{ is } S_2\}\] 
is open in $X$ and contains all codimension $\leq 1$ points of $X$. In particular, the locus where $X$ is $S_2$ is open and contains all codimension $\leq 1$ points. \end{lemma}

\begin{proof} The question is \'etale local so we may and do assume that $X$ is a scheme. By \cite[Cor. 6.11.7]{EGAIV.2} it suffices to show that every generic point of $X$ which belongs to $\mathbf{U}_{S_2}(\sh{F})$ admits an open neighborhood lying entirely in $\mathbf{U}_{S_2}(\sh{F})$. Indeed, this will show $\mathbf{U}_{S_2}(\sh{F})$ is open and since $\sh{F}$ is $S_1$, the set $\mathbf{U}_{S_2}(\sh{F})$ contains all codimension $\leq 1$ points. 

Thus, by reducedness and the assumptions we may replace $X$ with a dense open regular subset. However, since $\sh{F}$ is supported everywhere we may shrink $X$ further and assume $\sh{F}$ is locally free. Thus it suffices to show that $\mathcal{O}_X$ is $S_2$ on $X$, but since the local rings of $X$ are regular, this follows. \end{proof}

\begin{Theorem} \label{thm:singlecoherent} Let $X$ be a noetherian algebraic space which is separated, reduced and which satisfies the assumptions in \ref{conditions}. Suppose $\sh{F}$ is a coherent sheaf on $X$. If $x_1,...,x_r \in X$ is a finite set of points in $X$ which admits a common affine open neighborhood, then there is an open subset $U \subset X$ with the following properties:
\begin{enumerate}
\item $U$ contains $x_1,...,x_r$,
\item $U$ contains all points of codimension $\leq 2$ in $X$, 
\item $U$ admits a locally free sheaf $\sh{V}$, and 
\item a surjection $\sh{V} \to \sh{F}|_U \to 0$.
\end{enumerate}
\end{Theorem} 

\begin{proof}: By Lemma \ref{S_1}, we may assume that $\sh{F}$ is $S_1$ and supported everywhere. Since $\sh{F}$ is $S_1$, the set $\mathbf{U}_{S_2}(\sh{F})$ is open and contains all codimension $\leq 1$ points of $X$ by Lemma \ref{stwocoherent}. This implies the set of codimension $2$ points $p \in X$ where $\text{depth}(\mathcal{F}_p)<2$ is finite, denote them by $p_1,...,p_n$. 

%The locus $\textbf{U}_{S_2}(\sh{F})$ of $X$ where $\sh{F}$ is not $S_2$ is open by \cite[Prop. 6.11.6]{EGAIV.2} and because $\sh{F}$ is $S_1$, this locus is the complement of a finite union of codimension $\geq 2$ irreducible closed subschemes of $\text{Supp}(\sh{F})=X$. Let $p_1,...,p_n$ denote the generic points of this locus which have codimension $2$.

Take $\Spec A \subset X$ to be a dense affine open subset which contains all the singular codimension $1$ points of $X$ as well as $x_1,...,x_r$. Indeed, since $X$ satisfies the conditions in \ref{conditions} and is reduced, this is a finite set and we may apply Corollary \ref{cor:codim1} to produce such an open affine subset. Let $D$ denote the reduced complement of $\Spec A$ and consider the points in the intersection 
\[(\{p_1,...,p_n\} \cup |\text{Sing}(X)|) \cap |D|\]
which have codimension $\leq 2$ in $X$, call it $Z$. Since $\Spec A$ contains all singular codimension $1$ points, $Z$ must be a finite set $Z=\{z_1,...,z_n\}$. 

If $y_0 \in D$ is any point which has codimension $\leq 2$ \emph{viewed as a point in} $X$, then by Corollary \ref{cor:codim1} there is a dense affine open $\Spec B_0 \subset D$ which contains $Z \cup \{y_0\}$. Observe that the open subset in $X$ defined as $U_{0}=\Spec A \cup \Spec B_0$ satisfies the following properties:
\begin{enumerate}
\item $U_{0}$ contains $x_1,...,x_r,y_0$,
\item $U_{0}$ contains all codimension $\leq 1$ points of $X$,
\item $U_{0}$ contains all singular codimension $\leq 2$ points of $X$, and
\item $U_{0}$ contains all codimension $2$ points $p$ of $X$ where $\mathrm{depth}(\sh{F}_p)<2$.
\end{enumerate}

Note that $U_0$ satisfies the hypothesis of Corollary \ref{MVresolution} by Corollary \ref{cor:mvsquaresexist} and so $U_0$ has the resolution property. In particular, there is a locally free sheaf, $\sh{V}_0'$, on $U_0$ and a surjection
\[\sh{V}_0' \to \sh{F}|_{U_0}\]
If we push this morphism along the immersion $i_0: U_0 \to X$ we obtain a map
\[\sh{V}_0=(i_0)_*\sh{V}_0' \to (i_0)_*\sh{F}|_{U_0}\]
Some observations are in order.
\begin{enumerate}
\item The sheaf $\sh{V}_0=(i_0)_*\sh{V}_0'$ is coherent in a neighborhood of every codimension $\leq 2$ point by a result of Koll\'ar (see \cite[Tag 0BK3]{stacks} or \cite{MR3694300}). Indeed, $\sh{V}_0$ is coherent on $U_0$ because it is locally free there, hence near every codimension $\leq 1$ point and every singular codimension $2$ point so it suffices to show this near regular codimension $2$ points. Thus we may assume $X$ is the spectrum of a regular local ring $R$ of dimension $2$ and $U_0$ is the punctured spectrum. Now $\sh{V}_0'$ is locally free so only has the generic point as an associated point, hence the aforementioned lemma applies because $\hat{R}$ is regular of dimension $2$ as well (see \cite[Tags 07NV, 07NY]{stacks}). It follows that after possibly removing a closed codimension $\geq 3$ subset of $X$ disjoint from $U_0$, $\sh{V}_0$ is a coherent sheaf.

\item The sheaf $\sh{V}_0$ is locally free of finite rank at every codimension $\leq 2$ point. As above, we only need to check this at regular codimension $2$ points so we may assume $X$ is the spectrum of a regular local ring $R$ of dimension $2$ and $U_0$ is the complement of the closed point $p$. Since $\sh{V}_0$ is the pushforward of a sheaf on $U_0$ and is coherent, we have $H^i_p(X,\sh{V_0})=0$ for $i=0,1$, so $\sh{V}_0$ has depth $2$ by \cite[Thm. 3.8]{MR0224620}. Thus, $\sh{V}_0$ has projective dimension $0$ by the Auslander-Buchsbaum formula (see, for instance, \cite[Tag 090V]{stacks}). In conclusion, after possibly throwing out a closed subset of codimension $\geq 3$ from $X$ which is disjoint from $U_0$, we may assume $\sh{V}_0$ is locally free.

\item The natural map $h: \sh{F} \to (i_0)_*\sh{F}|_{U_0}$ is an isomorphism at every codimension $\leq 2$ point. Indeed, it is true at every point in $U_0$ and so is true at every codimension $\leq 1$ point of $X$. Moreover, since $U_0$ contains all codimension $2$ points where $\sh{F}$ is not $S_2$, it remains to be checked at codimension $2$ points $p$ where $\text{depth}(\sh{F}_p)=2$. Thus we may assume $X$ is the spectrum of a local ring of dimension $2$ where $\sh{F}$ is $S_2$ and $U_0$ is the punctured spectrum. Now the result follows by applying the standard long exact sequence in local cohomology to $\sh{F}$ with respect to the closed point $p \in X$ and \cite[Thm. 3.8]{MR0224620}. Namely, it shows that 
\[H^0(X,\sh{F}) \xrightarrow{\sim} H^0(U_0,\sh{F})=H^0(X,(i_0)_*\sh{F}|_{U_0})\]
which implies $h$ is an isomorphism. It follows that we may identify $\sh{F}$ with $(i_0)_*\sh{F}|_{U_0}$ after perhaps throwing out a closed subset of codimension $\geq 3$ which is disjoint from $U_0$.

\end{enumerate}
To summarize, after throwing out a codimension $\geq 3$ subset of $X$ which is disjoint from $U_0$, we obtain a map of sheaves 
\[f_0: \sh{V}_0 \to \sh{F}\] 
satisfying the following properties.
\begin{enumerate} 
\item $\sh{V}_0$ is a locally free sheaf, 
\item $f_0$ is surjective away from the closure of finitely many codimension $2$ points of $X$, and
\item $f_0$ is surjective at $x_1,...,x_r,y_0$.
\end{enumerate}

We are not finished because the cokernel of $f_0$ may be supported on codimension $2$ points of $X$. However, since $U_0$ contains all codimension $\leq 1$ points, the set of such codimension $2$ points is finite and we denote them by $y_1,...,y_l \in X$. Now, we may repeat the procedure above but with $y_i$ instead of $y_0$ to obtain locally free sheaves $\sh{V}_{i}$, and a surjective map
\[\bigoplus_{i=0}^l f_{i}:\bigoplus_{i=0}^l \sh{V}_{i} \to \sh{F}\]
after throwing out a closed subset of codimension $\geq 3$ from $X$ which is disjoint from $\{x_1,...,x_r\}$. \end{proof}

\section{Strong tensor generators on algebraic spaces} \label{strong}
To establish the resolution property one has to show that \emph{every} coherent sheaf is the quotient of a locally free sheaf. The purpose of this section is to reduce to the case of a \emph{single} coherent sheaf and thereby verify the conjecture of Gross using the results of the previous section.

Let $\mathbb{N}[s,t]$ denote the set of polynomials in $s$ and $t$ which have non-negative integer coefficients. Note that if $p(s,t) \in \mathbb{N}[s,t]$ is a polynomial $p(s,t)=\Sigma_{m,n} a_{m,n}s^mt^n$ and $\sh{F}$ is a quasi-coherent sheaf of finite presentation then then we can form a new quasi-coherent sheaf of finite presentation
\[p(\sh{F},\sh{F}^{\vee})=\bigoplus_{m,n} (\sh{F}^{\otimes m} \otimes (\sh{F}^{\vee})^{\otimes n})^{\oplus a_{m,n}}.\]

\begin{definition} Let $X$ be an algebraic space and suppose $\sh{F}$ is a quasi-coherent sheaf of finite presentation.
\begin{enumerate}
\item We say $\sh{F}$ is a \emph{strong tensor generator} if for every quasi-coherent sheaf of finite type $\sh{G}$ there is a polynomial $p(s,t) \in \mathbb{N}[s,t]$ and a surjection $p(\sh{F},\sh{F}^{\vee}) \to \sh{G}$.  
\item We say $\sh{F}$ is a \emph{resolving sheaf} if for every open subset $U$ such that $\sh{F}|_U$ is the quotient of a locally free sheaf, the open set $U$ has the resolution property. \end{enumerate} \end{definition}

\begin{remark} Gross defines \emph{strong tensor generators} only in the context of locally free sheaves so our definition generalizes his (see \cite[Definition 5.1]{GrossRes}). However, he describes the weaker notion of \emph{tensor generators} (also confined to the locally free case) and we refer to \cite[Def. 5.1]{GrossRes} for the definition. \end{remark}

\begin{proposition} \label{prop:strongresolve} Let $X$ denote a noetherian algebraic space and suppose that $\sh{F}$ is a strong tensor generator, then $\sh{F} \oplus \sh{F}^{\vee}$ is a resolving sheaf. \end{proposition} 

\begin{proof} Let $j: U \subset X$ be an open immersion and suppose that there is a locally free sheaf $\sh{V}$ on $U$ along with a surjection $\sh{V} \to (\sh{F} \oplus \sh{F}^{\vee})|_U$. We will show that $\sh{V}$ is a locally free strong tensor generator on $U$. This will imply that the frame bundle of $\sh{V}$ is quasi-affine and hence that $U$ has the resolution property, as desired (see \cite[Thm. 1.1]{GrossRes}).

Let $\sh{M}_U$ be a coherent sheaf on $U$, then by \cite[Cor. 15.5]{champs}, there exists a coherent subsheaf $\sh{M}' \subset j_*\sh{M}_U$ which extends $\sh{M}_U$. Now since $\sh{F}$ is a strong tensor generator, there is a polynomial $p(s,t) \in \mathbb{N}[s,t]$ and a surjection
\[p(\sh{F},\sh{F}^{\vee}) \to \sh{M}' \to 0\]
and upon restriction to $U$ we obtain a composition of surjections
\[p(\sh{V}, \sh{V}) \to p(\sh{F}|_U,\sh{F}^{\vee}|_U)=p(\sh{F},\sh{F}^{\vee})|_U \to \sh{M}'|_U=\sh{M}_U \to 0\]
which shows that $\sh{V}$ is a strong tensor generator on $U$. \end{proof}

%\begin{remark} Note that there are resolving sheaves which are not strong tensor generators. Indeed, a tensor generator (see \cite[Def. 5.1]{GrossRes}) which is not a strong tensor generator is such an example. For a specific instance of this configuration  \end{remark}

In other words, if $X$ admits a strong tensor generator, then it admits a resolving sheaf. Schr\"oer and Vezzosi in \cite[Prop. 2.2]{NormalSurfaceRes} (see also \cite[Ex. 1.3]{GrossRes}) proved that any noetherian scheme admits a coherent strong tensor generator. Unfortunately, the argument relies on the existence of an affine open covering. In the following, we show that normal algebraic spaces always admits an $S_2$ resolving sheaf $\ms{F}$. In the absence of the normality hypothesis, we find a coherent tensor generator after removing a very small subset as long as we are in characteristic $0$. We begin with the following refinement of \cite[Tag 0BD1]{stacks}. 

\begin{lemma} \label{lem:quasiaffine} Fix a quasi-compact and quasi-separated algebraic space $Y$. Suppose there is a finite surjective morphism $\pi: X \to Y$ of algebraic spaces where $Y$ and $X$ are integral, and $Y$ is normal. Then if $X$ is a quasi-affine scheme, $Y$ must also be quasi-affine. \end{lemma}

\begin{proof} Since $X$ is quasi-affine it is an AF-scheme and hence $Y$ is an AF-scheme by \cite[Thm. 1.2]{GrossRes}. We are thus reduced to the case of schemes and we may conclude by \cite[Tags 0BD1, 0BD3]{stacks}. \end{proof}

\begin{proposition} \label{prop:singlesheaf} Let $X$ be a separated noetherian algebraic space and suppose $x_1,...,x_r \in X$ is a collection of points which admits a common affine neighborhood. 

\begin{enumerate} 
\item If $X$ is normal, then there is a coherent $S_2$ resolving sheaf $\sh{F}$ on $X$ with $\mathrm{Supp}(\sh{F})=X$. \item If $X$ lies over $\mathbb{Q}$ and its normalization remains noetherian, there exists a closed subspace $C \subset X$ of codimension $\geq 3$ not containing $x_1,...,x_r$, such that the complement $U=X \smallsetminus C$ has a strong tensor generator $\sh{F}_U$.
\end{enumerate}

\end{proposition}

\begin{proof} If $X$ is normal, then there is a normal scheme $Y$ equipped with an action of a finite abstract group $G$ and an isomorphism $X=Y/G$ (see \cite[Cor. 16.6.2]{champs}). In particular, there is a finite morphism $\pi: Y \to X$. We know that $Y$ admits a strong tensor generator $\sh{G}$ which is $S_2$ everywhere and locally free on the regular locus of $Y$ by the proof of \cite[Prop. 2.2]{NormalSurfaceRes}. Indeed, there it is shown that every coherent sheaf $\mathcal{M}$ is the quotient of a (finite) direct sum $\bigoplus_i \mathcal{O}(-t_iD_i)$ where $D_i$ is a Weil divisor, $t_i \in \mathbb{N}$ and $\{X\smallsetminus D_i\}$ is an affine open cover of $Y$ (see also \cite[Tag 0F89]{stacks}).

We claim that $\sh{F}=\pi_*\sh{G}$ is a $S_2$ resolving sheaf. To show it is resolving, assume $\sh{V}$ is a locally free sheaf on an open subset $U \subset X$ and we have a surjection $\sh{V} \to \sh{F}|_U$, then we obtain another surjection on $Y_U=Y \times_X U$:
\[\pi_U^*\sh{V} \to \pi_U^*\sh{F}|_U \to \sh{G}|_U\]
because $\pi$, and hence $\pi_U$, is affine. Thus, the frame bundle of $\pi^*\sh{V}$, call it $W$, is quasi-affine because $\pi^*\sh{V}$ is a strong tensor generator on $Y_U$ (see \cite[Thm. 6.4]{GrossRes}) and we have the following diagram all of whose squares are cartesian
\[\begin{tikzcd}
  W=\text{Fr}(\pi^*\sh{V}) \arrow[d] \arrow[r] & Y_U \arrow[d] \arrow[r] & Y \arrow[d]  \\
  \left [W/G\right ] \arrow[d] \arrow[r] & \left[Y_U/G\right ] \arrow[r] \arrow[d] & \left [Y/G\right ] \arrow[d]  \\
  \text{Fr}(\sh{V}) \arrow[r] & U \arrow[r] & Y/G=X 
  \end{tikzcd}
\]
Indeed, the three top vertical arrows are $G$-torsor maps since the rightmost one is and this is stable under base change. The bottom three vertical arrows are coarse moduli space maps since the rightmost one is and this is stable under flat base change. Thus we see that $W/G$ can be identified with the frame bundle of $\sh{V}$. Next, observe that $W \to W/G$ is a finite surjective morphism between normal algebraic spaces, since it is proper and quasi-finite. Hence by Corollary \ref{lem:quasiaffine}, the fact that $W$ is quasi-affine implies $W/G$ is. Therefore the frame bundle of $\sh{V}$ is quasi-affine. In particular, $U$ has the resolution property by \cite[Thm. 1.1]{TotaroRes}.

It remains to see why $\sh{F}=\pi_*\sh{G}$ is $S_2$ but this follows because $\sh{G}$ is supported everywhere and $S_2$ (see \cite[Cor. 5.7.11 (ii)]{EGAIV.2}). Lastly, because $\pi$ is dominant, $\text{Supp}(\sh{G})=Y$ implies $\text{Supp}(\sh{F})=X$.

If $X$ lies over $\Spec \mathbb{Q}$, first note that there is a dense open set $U \subset X$ so that $U$ contains the $x_1,...,x_r$ as well as all codimension $\leq 2$ points and where we may write $U=\bigcup_{i=1}^m U_i$ such that each $U_i$ has the resolution property. To see this, fix a $p \in X$ of codimension $\leq 2$ and by Proposition \ref{prop:affineMV} there is a dense open subset $W_p \subset X$ containing $x_1,...,x_r,p$ and which satisfies the hypothesis of Proposition \ref{MVresolution}. Hence each $W_p$ has the resolution property. Thus, we may take $U= \bigcup_{p} W_{p}$ where $p$ runs through the codimension $\leq 2$ points of $X$. By the noetherian hypothesis, there are finitely many $W_p$ which cover $U$.

This means that after replacing $X$ with $U$, we may suppose that $X$ is covered by (finitely many) dense open subsets $U_i$ each of which has the resolution property. For each $1 \leq i \leq m$, let $\sh{F}_i$ denote a coherent extension (to $X$) of a (locally free) strong tensor generator for $U_i$. Indeed, a locally free tensor generator exists on $U_i$ by \cite[Thm. 1.1]{GrossRes} and in characteristic $0$, all tensor generators are strong tensor generators by \cite[Prop. 6.5]{GrossRes}.

Now, if $\sh{I}_i$ denotes the ideal sheaf of the reduced complement $X \smallsetminus U_i$ then we claim that 
\[\sh{G}=\bigoplus_{i=1}^n \sh{I}_i \oplus \sh{F}_i \oplus \sh{F}_i^{\vee}\]
is a strong tensor generator. Indeed, if $\sh{H}$ is an arbitrary coherent sheaf, we know that there is a surjection 
\[f_i: p_i(\sh{F}_i|_{U_i}, \sh{F}_i^{\vee}|_{U_i}) \rightarrow \sh{H}|_{U_i}\] 
for some polynomial $p_i(s,t) \in \mathbb{N}[s,t]$. Using \cite[Prop. 6.9.17]{EGAI.2}, there is a $n_i$ and a map 
\[\sh{I}_i^{n_i}p_i(\sh{F}_i, \sh{F}_i^{\vee}) \rightarrow \sh{H}\]
which extends $f_i$ and is therefore surjective on $U_i$. Doing this for every $1 \leq i \leq m$, we obtain a surjection 
\[\bigoplus_{i=1}^m f_i: \bigoplus_{i=1}^m \sh{I}_i^{n_i}p_i(\sh{F}_i, \sh{F}_i^{\vee}) \rightarrow \sh{H}.\] 

\noindent Moreover one can see that there is a surjection 
\[\bigoplus_{i=1}^m \sh{G}^{\otimes n_i} \otimes p_i(\sh{G}, \sh{G}) \rightarrow \bigoplus_{i=1}^m \sh{I}_i^{n_i}p_i(\sh{F}_i, \sh{F}_i^{\vee}).\]

\noindent Thus every coherent sheaf on $X$ is a quotient of a polynomial in $\sh{G}$ and this implies $\sh{G}$ is a strong tensor generator, as desired. \end{proof}

\begin{Remark} We do not know if Proposition \ref{prop:singlesheaf} holds in the absence of the characteristic $0$ hypothesis. Indeed, for the proof to go through we need to show that non-normal algebraic spaces which satisfy the hypothesis of Corollary \ref{cor:mvsquaresexist} admit strong locally free tensor generators. However, we only know that locally free tensor generators exist on such spaces. \end{Remark}

We now use this to verify the first part of the conjecture of Gross (see \cite[Conj. 2.0.3]{Grossthesis}) when $X$ is reduced. In fact, using the machinery developed thus far, we are even able to establish this in certain non-schematic cases as well. The following result will allow us to remove the reducedness hypothesis in positive characteristic.

\begin{proposition} \label{positivechar} Suppose $X$ is a noetherian algebraic space that lies over $\mathbb{F}_p$. Then $X_{\mathrm{red}}$ has the resolution property if and only if $X$ has the resolution property. \end{proposition}

\begin{proof} If $X$ has the resolution property, then the pushforward of any coherent sheaf on $X_{\mathrm{red}}$ is the quotient of a locally free sheaf on $X$. By restricting, we see that $X_{\text{red}}$ has the resolution property. 

For the other direction, note that the closed immersion $i: X_{\text{red}} \to X$ is defined by a nilpotent ideal sheaf $\sh{I} \subset \mathcal{O}_X$. Then because there is a $n>0$ with $\sh{I}^{p^n}=0$ we obtain a commutative diagram
$$
\begin{tikzcd} 
						& X\ar[ld, "f"] \ar[d,"F^n"]\\	 
X_{\text{red}} \ar[r]	& X
\end{tikzcd}
$$
where $F^n$ is the $n$-fold iterated (absolute) Frobenius. Moreover, the composition $f \circ i=X_{\text{red}} \to X \to X_{\text{red}}$ is also the $n$-fold iterated (absolute) Frobenius on $X_{\text{red}}$. 

Let $\sh{V}$ denote a locally free sheaf on $X_{\text{red}}$ whose frame bundle is quasi-affine, this exists by our hypotheses and \cite[Thm. 1.1]{GrossRes}. To show $X$ has the resolution property it suffices to show that $f^*\sh{V}$ also has quasi-affine frame bundle. But since $i$ is a nilpotent immersion it suffices to show that the frame bundle of $i^*f^*(\sh{V})=(F_{X_{\text{red}}}^n)^*\sh{V}$ is quasi-affine. Indeed, if $W' \to W$ is a nilpotent immersion of algebraic spaces where $W'$ is quasi-affine then by \cite[Prop. 3.5]{GrossRes} $W$ is a scheme and \cite[Tag 0B7L]{stacks} implies $W$ is also quasi-affine. 

Thus, we are reduced to checking that if the locally free sheaf $\sh{V}$ has a quasi-affine frame bundle, then $(F_{X_{\text{red}}}^n)^*\sh{V}$ does as well. But the induced map on frame bundles $W' \to W$ is finite and since $W$ is quasi-affine, it follows that $W'$ is quasi-affine as well. This shows that $X$ has the resolution property. \end{proof}

\begin{Theorem} \label{resolutionthree} Let $X$ be a noetherian and separated algebraic space which satisfies the conditions in \ref{conditions} and is either reduced or lies over $\mathbb{F}_p$. Let $x_1,...,x_r \in X$ denote a set of points which admits a common affine open neighborhood and suppose one of the following holds:\begin{enumerate}
\item $X$ is a scheme, or
\item $X$ is a normal algebraic space, or
\item $X$ lies over $\mathbb{Q}$.
\end{enumerate}
Then there exists an open subset $U \subset X$ with the resolution property which contains every point of codimension $\leq 2$ in $X$ together with the points $x_1,...,x_r$.
 \end{Theorem} 

\begin{proof} By Proposition \ref{positivechar}, it suffices to assume that $X$ is reduced. By Propositions \ref{prop:singlesheaf} and \ref{prop:strongresolve}, a resolving sheaf $\sh{F}$ exists on $X$ after perhaps removing a closed codimension $\geq 3$ subset which does not contain any of the $x_1,..,x_r$. Then, by Theorem \ref{thm:singlecoherent}, there is an open subset $U \subset X$ which contains $x_1,...,x_r$ and every point of codimension $\leq 2$, a locally free sheaf $\sh{V}$ on $U$, and a surjection
\[\sh{V} \to \sh{F}|_U \to 0,\]
as desired. Since $\sh{F}$ is a resolving sheaf, this shows $U$ satisfies the resolution property, as desired. \end{proof}

To conclude, we verify the second part of Gross' conjecture (see \cite[Conj. 2.0.3]{Grossthesis}) under some mild hypotheses.

\begin{corollary} \label{surjectionthree} Let $X$ denote a noetherian scheme which satisfies the assumptions of \ref{conditions} and at least one of the following conditions:
\begin{enumerate} 
\item $X$ is reduced, or
\item $X$ lies over a field of positive characteristic.
\end{enumerate}
If $\sh{M}$ is a coherent sheaf on $X$, then there is a surjection $\sh{F} \to \sh{M}$ where $\sh{F}$ is coherent and $\sh{F}$ is locally free away from a closed subset of codimension $\geq 3$. 
\end{corollary}

\begin{proof} For each $x_{\lambda} \in X$, Theorem \ref{resolutionthree} implies there is an open neighborhood $U_{\lambda} \subset X$ with the resolution property and with $\text{codim}(X \smallsetminus U_{\lambda}) \geq 3$. Thus, it follows that there is a surjection $f_{\lambda}: \sh{F}_{\lambda}' \to \sh{M}|_{U_{\lambda}} \to 0$ where $\sh{F}_{\lambda}'$ is a locally free sheaf on $U_{\lambda}$. Let $\sh{F}_{\lambda}''$ denote a coherent extension of $\sh{F}_{\lambda}'$ to all of $X$. Then, applying \cite[Prop. 6.9.17]{EGAI.2} to the inclusion $i: U_{\lambda} \to X$, there is an integer $n_{\lambda}$ and a morphism
\[\sh{I}_{\lambda}^{n_{\lambda}}\sh{F}_{\lambda}'' \to \sh{M}\]
extending $f_{\lambda}$, where $\sh{I}_{\lambda}$ is the ideal sheaf of $Z_{\lambda}=X \smallsetminus U_{\lambda}$. Thus, for each $x_{\lambda} \in X$, there is a coherent sheaf $\sh{F}_{\lambda}=\sh{I}_{\lambda}^{n_{\lambda}}\sh{F}_{\lambda}''$, a morphism $\sh{F}_{\lambda} \to \sh{M}$ which is surjective near $x_{\lambda}$, with $\sh{F}_{\lambda}$ locally free away from a closed subset of codimension $\geq 3$. Since $X$ is noetherian, there are finitely many such sheaves $\sh{F}_{1},...,\sh{F}_{n}$ on $X$ so that 
\[\sh{F}=\bigoplus_{i=1}^n \sh{F}_{i} \to \sh{M} \to 0\]
is surjective. Moreover, $\sh{F}$ is locally free away from a closed subset of codimension $\geq 3$. \end{proof}

\bibliography{mybib}{}
\bibliographystyle{plain}

\end{document}